\documentclass[10pt,fleqn]{article}
\usepackage{mathptmx}
\usepackage{color}

\input{epsf.tex}
\usepackage{latexsym,amsfonts,amssymb,epsfig,verbatim}
\usepackage{amsmath,amsthm,amssymb,latexsym,graphics,textcomp}
\usepackage{mathrsfs}
\usepackage{eucal}
\usepackage{graphicx}
\usepackage{url}

\input xy
\xyoption{all}

\theoremstyle{plain}

\newtheorem{theorem}{Theorem}

\newtheorem{proposition}[theorem]{Proposition}

\theoremstyle{definition}

\newtheorem*{question}{Question}

\newcommand{\ML}{\mathcal{ML}} 

\newcommand{\AH}{\mathrm{AH}}
\newcommand{\GF}{\mathrm{GF}}
\newcommand{\QF}{\mathrm{QF}}

\newcommand{\SL}{\mathrm{SL}}
\newcommand{\PSL}{\mathrm{PSL}}

\newcommand{\SU}{\mathrm{SU}}

\newcommand{\A}{\mathcal A}
\newcommand{\R}{\mathbb R}
\newcommand{\Z}{\mathbb Z}
\newcommand{\N}{\mathbb N}
\newcommand{\C}{\mathbb C}
\newcommand{\D}{\mathcal D}
\newcommand{\Schw}{\mathcal S}
\newcommand{\T}{\mathcal T}
\newcommand{\U}{\mathcal U}
\newcommand{\Y}{\mathcal Y}
\newcommand{\X}{\mathcal X}
\newcommand{\V}{\mathcal V}
\newcommand{\Q}{\mathcal Q}
\newcommand{\B}{\mathcal B}

\newcommand{\Hom}{\mathrm{Hom}}

\newcommand{\rmd}{\mathrm{d}}
\newcommand{\co}{\colon\thinspace}

\newcommand{\hol}{\mathrm{hol}}

\newcommand{\gr}{\mathrm{gr}}
\newcommand{\Gr}{\mathrm{Gr}}

\begin{document}

\title{\textbf{Slicing, skinning, and grafting}}
\author{David Dumas and Richard P. Kent IV\thanks{Both authors supported by NSF postdoctoral fellowships.}}

\date{March 14, 2008}

\maketitle



\noindent Let $M$ be a compact manifold with boundary. 
If $M$ is connected, let
$\X_\C(M)$
be the $\SL_2(\C)$--character variety of $M$. 
If not, take $\X_\C(M)$ to be the cartesian product of the character varieties of its components. 

Throughout, $S$ is a 
\textit{closed} 
connected 
oriented hyperbolic surface,  $\T(S)$ its Teich-m\"uller space.
By the Uniformization Theorem, $\T(S)$ is both the space of marked conformal structures on $S$ and the space of marked hyperbolic structures on $S$; we blur the distinction between these two descriptions, letting context indicate the desired one.

The variety $\X_\C(S)$ contains the space $\AH(S)$ of hyperbolic structures on $S \times \R$. By the work of A. Marden \cite{marden} and D. Sullivan \cite{sullivanstability}, the interior of $\AH(S)$ is the space of quasifuchsian groups $\QF(S)$, and $\QF(S)$ lies in the smooth part of $\X_\C(S)$---though $\AH(S)$ sits more naturally in the $\PSL_2(\C)$--character variety of $S$, and has many lifts to the variety $\X_\C(S)$, we content ourselves with $\X_\C(S)$, as our arguments apply to any lift considered. 
We refer the reader to \cite{PSL} for a detailed treatment of the $\PSL_2(\C)$--character variety. 
The quasifuchsian groups are parameterized by the product of Teich-m\"uller spaces $\T(S) \times \T(\overline S)$, by the Simultaneous Uniformization Theorem of L. Bers \cite{berssimul},  and the \textbf{Bers slice} $\B_Y$ is the set
\[
\B_Y = \T(S) \times \{ Y \} \subset \X_\C(S) .
\] 
As we will see, a Bers slice is cut out of $\X_\C(S)$ by an analytic subvariety of dimension $-\frac{3}{2}\chi(S)$. It is never cut out by an algebraic subvariety:
\begin{theorem}\label{slice} Let $\V \subset \X_\C(S)$ be a
complex algebraic subvariety of dimension $-\frac{3}{2} \chi(S)$. 
Then the Bers slice $\B_Y$ is not contained in $\V$.
\end{theorem}

The proof of the theorem says nothing more about the Zariski closure of $\B_Y$, and, as it is a theorem of W. Goldman \cite{goldmancomponents} that $\Hom\big(\pi_1(S), \SL_2(\C)\big)$, and hence $\X_\C(S)$, is irreducible, it is natural to wonder:

\begin{question} Is $\B_Y$ Zariski--dense in $\X_\C(S)$?
\end{question}

If we drop the requirement that $S$ be closed, and instead ask only that $S$ be of finite volume, we suspect that Theorem \ref{slice} still holds, where $\X_\C(S)$ now denotes the variety of $\SL_2(\C)$--characters of representations that are parabolic on peripheral subgroups.
At a key point, our argument appeals to properness of the holonomy map $\Q(Y) \to \X_\C(S)$ from the space of quadratic differentials to $\X_\C(S)$, see section \ref{proj}, which remains unknown when $S$ is noncompact, and so this becomes an obstacle when attempting to generalize Theorem \ref{slice}.

\bigskip 
\noindent Let $M$ be a compact oriented $3$--manifold with connected incompressible boundary whose interior admits a complete hyperbolic metric.  The Simultaneous Uniformization Theorem admits a generalization due to L. Ahlfors, Bers, Marden, and Sullivan, see \cite{bersparameter}: the minimally parabolic geometrically finite hyperbolic structures on $M^\circ$ compatible with the orientation on $M$ are parameterized by the Teichm\"uller space of $\partial M$. 
There is a map
\[
\GF(M) \cong \T(\partial M) \to \T(\partial M) \times \T(\overline{\partial M}) \cong \QF(\partial M) 
\]
induced by inclusion and given by
\[
X \mapsto (X, \sigma_M(X)).
\]
The map $\sigma_M \co \T(\partial M) \to \T(\overline{\partial M})$ is W. Thurston's \textbf{skinning map}, which arises in his proof of the Geometrization Theorem for Haken Manifolds, see \cite{kent} and \cite{skinmcmullen}.

Theorem \ref{slice} has the following corollary.

\begin{theorem}\label{skin} Let $M$ be a compact oriented $3$--manifold with incompressible boundary of negative Euler characteristic whose interior admits a complete hyperbolic metric without accidental parabolics. Then its skinning map $\sigma_M$ is not constant.
\end{theorem}
A hyperbolic structure on the interior of a compact oriented $3$--manifold $M$ has \textbf{accidental parabolics} if there is an element of the fundamental group of a nontorus component of $\partial M$ that is parabolic in the corresponding Kleinian group.

When the boundary of $M$ is disconnected, the skinning map is defined as follows.
Let $S_1 \cup \cdots \cup S_n$ be the union of the components of $\partial M$ that are not tori.  Each inclusion $S_j \to M$ induces a map
\[
\GF(M) \cong \T(S_1 \cup \cdots \cup S_n) = \T(S_1) \times \cdots \times \T(S_n) \longrightarrow \QF(S_j)
\] 
given by
\[
X \mapsto (X_j, \sigma_j(X)).
\]
The hypothesis that the interior of $M$ admits a hyperbolic metric without accidental parabolics is needed to guarantee that the range of this map lies in $\QF(S_j)$.
The skinning map 
\[
\sigma_M \co \T(S_1 \cup \cdots \cup S_n)  \longrightarrow \T(\overline{S_1}) \times \cdots \times \T(\overline{S_n})
\]
of $M$ is then defined to be
\[
\sigma_M (X) = (\sigma_1(X),  \ldots, \sigma_n(X)).
\]

Thurston's Bounded Image Theorem gives global constraints on $\sigma_M$. 
Namely, if $M$ is acylindrical and satisfies the hypotheses of Theorem \ref{skin}, the image of $\sigma_M$ is bounded, see \cite{kent}.  
C. McMullen proved \cite{skinmcmullen}, under the same hypotheses, that there is a constant $c < 1$ depending on $M$ such that the Teichm\"uller operator norm satisfies $\| \rmd \sigma_M \| < c$ over all of Teichm\"uller space. 
Aside from holomorphicity, McMullen's theorem, and the conclusion of Theorem \ref{skin}, 
little is known concerning the local behavior of $\sigma_M$.
For instance, the following question remains unanswered.
\begin{question} Are skinning maps always open?
\end{question}

\bigskip
\noindent \textbf{Acknowledgments.}  The authors thank Dan Abramovich and Nathan Dunfield for helpful discussions on complex algebraic geometry.

\section{The character variety}

The construction of the character variety discussed here may be found in full detail in \cite{cullershalen}, see also \cite{shalen}.

Let $\{ x_1, \ldots, x_n\}$ be a generating set for $\pi_1(M)$ and let 
\[
\{w_j\}_{j=1}^N = \{ x_{i_1} x_{i_2} \ldots \, x_{i_k} 
\ | \  1 \leq i_1 < i_2 < \cdots < i_k \leq n \} .
\]
For each $j$, we have the function
\[
I_{w_j} \co \Hom(\pi_1(M), \SL_2(\C) ) \to \C
\]
given by
\[
I_{w_j}(\rho) = \mathrm{trace}(\rho(w_j)).
\]
We thus obtain a map
\[
t \co \Hom(\pi_1(M), \SL_2(\C) ) \to \C^N
\]
given by 
\[
t(\rho) = \big( I_{w_1}(\rho), \ldots, I_{w_N}(\rho) \big),
\]
and the image of $t$ is a variety $\X_\C(M)$, called the \textbf{character variety} as it parameterizes the characters of representations of $\pi_1(M)$ into $\SL_2(\C)$.

Note that the set $\X_\R(M)$ of real points of $\X_\C(M)$ contains the image of \linebreak $\Hom(\pi_1(M), \SL_2(\R))$ under $t$.
In fact, it is a theorem of H. Bass, J. Morgan, and P. Shalen that any real character is the character of a representation into $\SL_2(\R)$ or $\SU(2)$---see Proposition III.1.1of \cite{morganshalen}.
The space $\QF(S)$ embeds into the smooth locus of $\X_\C(S)$ and 
contains only characters of discrete faithful representations.
Since the \linebreak Teichm\"uller space is properly embedded in $\X_\C(S)$, it follows that $\T(S) = \QF(S) \cap \X_\R(S)$ is a topological component of $\X_\R(S)$.

The variety $\X_\C(M)$ may be interpreted as the quotient
\[
\Hom(\pi_1(M), \SL_2(\C) ) / \! \! \! / \, \SL_2(\C)
\]
of geometric invariant theory, which means that $\X_\C(M)$ is an affine variety equipped with a regular function
\[
\Hom(\pi_1(M), \SL_2(\C) ) \to \X_\C(M)
\]
that induces an isomorphism
\[
\C[\X_\C(M)] \to \C[\Hom(\pi_1(M), \SL_2(\C) )]^{\SL_2(\C)} 
\]
We warn the reader that defining the variety $\X_\C(M)$ via geometric invariant theory only specifies $\X_\C(M)$ up to birational equivalence, and that we will always use the representative constructed above.

\section{Projective structures}\label{proj}

We refer the reader to \cite{gunning,GKM,tanigawa,shigatanigawa} for more detailed discussions of complex projective structures on Riemann surfaces.

A marked complex projective structure on $S$ is a marked conformal structure together with an atlas of conformal charts taking values in $\C \mathbb P^1$ whose transition functions are restrictions of M\"obius transformations. 
Let $\mathcal P(S)$ denote the space of all marked complex projective structures on $S$ and let 
\[
\pi \co \mathcal P(S) \to \T(S)
\]
denote the map that forgets the projective structure.

The space $\pi^{-1}(Y)$ of complex projective structures with underlying Riemann surface $Y$ may be parameterized by the space $\Q(Y)$ of holomorphic quadratic differentials $\varphi$ on $Y$, which we think of as holomorphic cusp forms on the unit disk---a complex projective structure $P$ on $Y$ has a developing map $d \co \Delta \to \C \mathbb P^1$, and the Schwarzian derivative $\Schw d$ of $d$ is the holomorphic quadratic differential associated to $P$. 

There is a holomorphic embedding
\[
\hol \co \Q(Y) \to \X_\C(S)
\]
sending a projective structure to the character of its holonomy representation, see \cite{kra,kra2}.  
As a quasifuchsian group in $\B_Y$ is the image of the holonomy representation of a projective structure on $Y$, the image of $\hol$ contains $\B_Y$.

We will need the following theorem of D. Gallo, M. Kapovich, and A. Marden, Theorem 11.4.1 of \cite{GKM}.
\begin{theorem}[Gallo--Kapovich--Marden]
The map $\hol$ is proper. \qed
\end{theorem}

\noindent Prior to the proof of this theorem, it was shown by H. Tanigawa that $\hol$ maps properly into the subset of $\X_\C(S)$ consisting of the irreducible characters \cite{tanigawaDiv}. It is worth noting that the map 
$
\hol \co \mathcal P(S) \to \X_\C(S)
$
is \textit{not} proper, see \cite{hejhal}.

As mentioned in the introduction, it is not known if this theorem holds whenever $S$ has finite volume.

\section{Grafting}

Let $X$ be a point in $\T(S)$. We let $\ML_{2\pi \Z}(S)$ denote the set of multicurves on $S$ with weights in $2 \pi \N$. 
Let $\lambda$ be in this set, realized geodesically in $X$, and let $\lambda_1$, $\ldots \ $, $\lambda_n$ be the components of $\lambda$ with weights $w_1$, $\ldots \ $, $w_n$.
We create a Riemann surface $\gr_\lambda(X)$ by cutting $X$ open along $\lambda$ and inserting the union of flat annuli
\[
\bigcup_{i=1}^n \ \lambda_n \times [0,w_n]
\]
and say that we have \textbf{grafted} $X$ along $\lambda$.

We need the following theorem of Tanigawa.

\begin{theorem}[Tanigawa \cite{tanigawa}] Let $\lambda$ be an element of $\ML_{2 \pi \Z}(S)$.  The map 
\[
\gr_\lambda \co \T(S) \to \T(S)
\]
is a diffeomorphism. \qed
\end{theorem}

There is a projective version of the grafting procedure due to B. Maskit \cite{maskit} that begins with the data $X$ in $\T(S)$ and $\lambda$ in $\ML_{2\pi \Z}(S)$ and produces a projective structure $\Gr_\lambda(X)$ on $S$,
 see \cite{kamishima,tanigawa}. 
This procedure is natural in the sense that the following diagram commutes.
\begin{equation}\label{diagram}
 \xymatrix{\T(S) \ar@/^2.1pc/[rr]|{\mathrm{\, id\, }}    \ar[rd]_{\ \ \gr_\lambda} \ar[r]^{\! \! \Gr_\lambda} 
 &   \mathcal P(S) \ar[d]^{\pi} \ar[r]^{\! \! \! \! \! \! \! \! \hol} & \X_\C(S)   \\
 & \T(S)  
    } 
\end{equation}
More generally, both types of grafting may be performed along any measured geodesic lamination, as shown by Thurston, though we warn the reader that in the general setting the identity map in the diagram becomes the ``$\lambda$--bending map," see \cite{earthquakemcmullen} and \cite{epsteinmarden}.

There are special projective structures---the ones with Fuchsian holonomy---called \textbf{Fuchsian centers}.
W. Goldman has characterized these in terms of grafting.

\begin{theorem}[Goldman \cite{goldman}] Every Fuchsian center is 
$\Gr_\lambda(X)$ for some $\lambda$ in $\ML_{2\pi \Z}(S)$ and $X$ in $\T(S)$.
\qed
\end{theorem}

The picture in Figure \ref{figure}, drawn by the first author, shows us some of the Fuchsian centers in a particular $\Q(Y)$.

\begin{figure}
\begin{center}
\fbox{\includegraphics[scale=0.55]{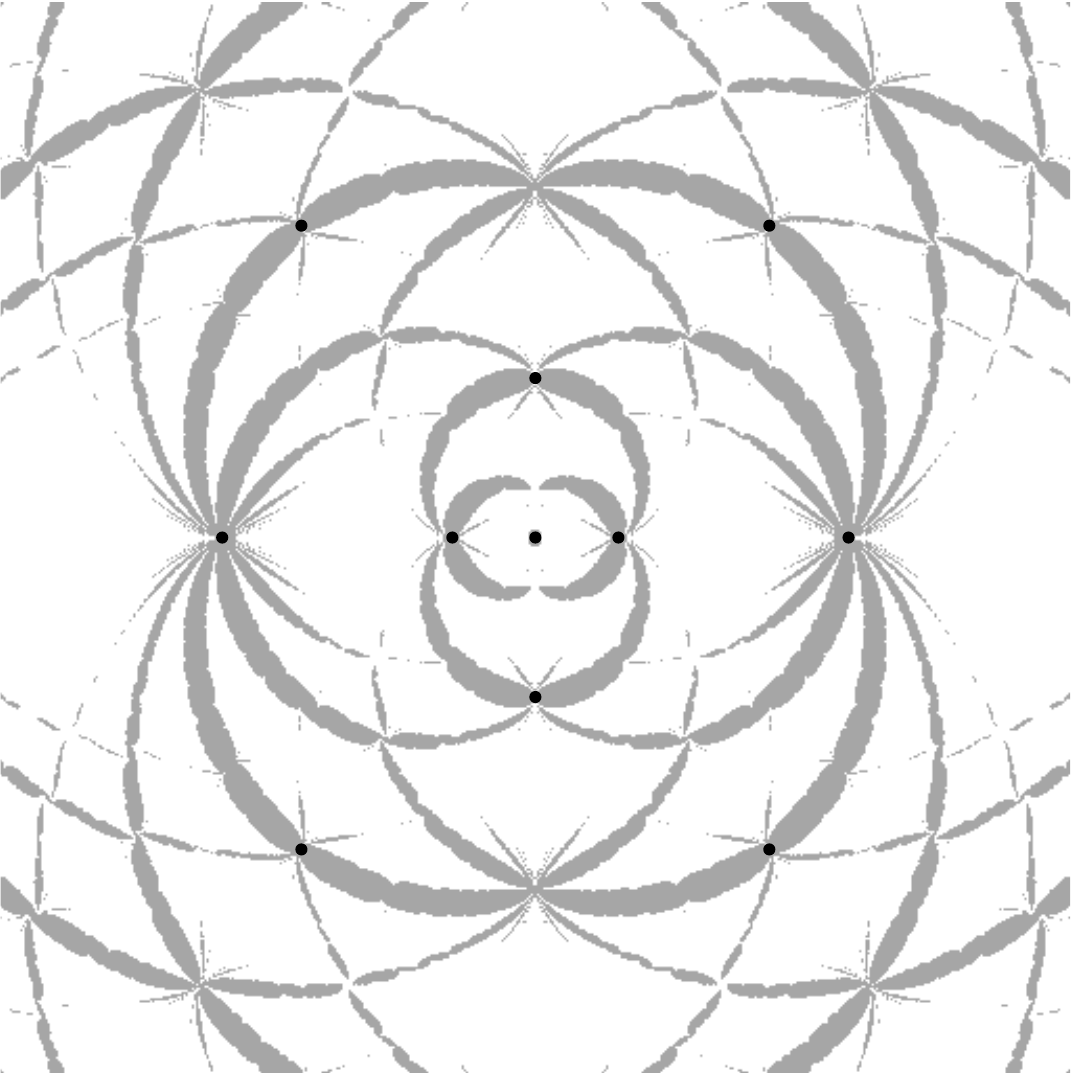}}
\caption{Part of the space of quadratic differentials on a square punctured torus $Y$, with $\hol^{-1}(\QF(S))$ in gray, and Fuchsian centers in black.}\label{figure}
\end{center}
\end{figure}

\bigskip \noindent To enumerate the Fuchsian centers in $\Q(Y)$, one proceeds as follows.

Tanigawa's theorem provides a map
\[
\Phi \co \ML_{2\pi \Z}(S) \to \Q(Y)
\]
defined by 
\[
\Phi(\lambda) = \Gr_\lambda ( \gr_\lambda^{-1} (Y) ),
\]
and Goldman's theorem with Tanigawa's and the diagram (\ref{diagram}) implies that 
\begin{equation}\label{intersection} 
\hol(\Q(Y)) \cap \T(S) = \hol \circ \Phi(\ML_{2\pi \Z}(S)) .
\end{equation}
To see this, note that 
$
\hol(\Q(Y)) \cap \T(S) \subset \X_\C(S)
$
is precisely the set of holonomy representations of Fuchsian centers in $\Q(Y)$. Goldman's theorem says that any Fuchsian center is $\Gr_\lambda(X)$ for some $\lambda$ and $X$, and if $\Gr_\lambda(X)$ is a projective structure on $Y$, then $X = \gr_\lambda^{-1}(Y)$, by (\ref{diagram}), and (\ref{intersection}) follows.

\begin{proposition}\label{infinite} The set $\, \hol(\Q(Y)) \cap \T(S)$ contains infinitely many points.
\end{proposition}
\begin{proof} By (\ref{intersection}), it suffices to show that the image of $\hol \circ \Phi$ is infinite.

Let $\{ n \}$ be a strictly increasing sequence in $2 \pi  \N$.
Let $\gamma$ be an essential simple closed curve and consider the sequence of projective structures $\Phi(n \gamma)$. 
Let $\gamma_n$ be the geodesic representative of $\gamma$ in $X_n = \gr_{n \gamma}^{-1} (Y)$.

The hyperbolic length $\ell_{X_n}(\gamma_n)$ of $\gamma_n$ in $X_n$ is tending to infinity.
To see this, note that if $\ell_{X_n}(\gamma_n)$ were bounded,
the conformal moduli of the annuli 
\[
\gamma_n \times [0,2\pi n] \subset Y = \gr_{n \gamma}(X_n)
\]
would be unbounded (see chapter one of \cite{ahlfors}), implying that the extremal length of $\gamma$ in $Y$ is zero, which is absurd.

So the $X_n$ are leaving the Teichm\"uller space.
Since $\hol \circ \Gr_{n \gamma}(X_n) = X_n$, the proof is complete. 
\end{proof}

In fact, by Thurston's theorem that $\Gr \co \ML(S) \times \T(S) \to \mathcal P(S)$ is a homeomorphism \cite{kamishima}, the function $\Phi$ is injective, see \cite{dumas}, though we will not need this here.

\section{Algebraic versus analytic geometry}

Let $\D$ be a domain in $\C \mathbb P^n$. A set $\A \subset \D $ is a \textbf{locally analytic set} if each point $a$ in $\A$ has a neighborhood $\U$ such that $\A \cap \U$ is the common set of zeros of a finite collection of holomorphic functions on $\U$. A locally analytic set in $\D$ is an \textbf{analytic set} in $\D$ if it is closed there.

The following theorem is well known.

\begin{theorem}\label{smooth} The set of smooth points of an irreducible complex affine or projective algebraic variety is connected in the classical topology.
\end{theorem}
\begin{proof} By passing to projective completions, it suffices to prove the theorem in the projective case.

Suppose to the contrary that $\V$ is an irreducible projective variety whose set of smooth points $\V_s$ is disconnected and write $\V_s = \U \sqcup \mathcal W$ with $\U$ and $\mathcal W$ nonempty open sets. 

The sets $\U$ and $\mathcal W$ are locally analytic sets.  Since the singular locus of $\V$ is an analytic set of dimension less than that of $\V$ 
(see Chapter II.1.4 of \cite{shafarevich1}, for example)
it follows from a theorem of R. Remmert and K. Stein \cite{remmertstein} (see also Chapter III of \cite{fromholo}) that the closures $\overline{\U}$ and $\overline{\mathcal W}$ of $\U$ and $\mathcal{W}$ in the classical topology are analytic sets.

A theorem of W.-L. Chow states that analytic sets in projective space are in fact algebraic sets \cite{chow}, and so $\V = \overline{\U} \cup \overline{\mathcal W}$ is a nontrivial union of proper algebraic subsets, contradicting the irreducibility of $\V$.
\end{proof}

\noindent
The use of Chow's theorem in the proof may be replaced with an application of the much stronger GAGA Principle of J.-P. Serre \cite{gaga}.

\section{Slicing}

\begin{proof}[Proof of Theorem \ref{slice}] Suppose to the contrary that $\V \subset \X_\C(S)$ is a subvariety of dimension $-\frac{3}{2} \chi(Y)$ containing $\B_Y$,
and let $\V_s$ denote the smooth part of $\V$. 

The complex dimension of $\Q(Y)$ is that of $\V$, and since $\hol(\Q(Y))$ contains $\B_Y$,  holomorphicity implies that $\hol(\Q(Y)) \subset \V$.
In fact, this demonstrates that $\hol(\Q(Y))$ must lie in an irreducible component of $\V$, and so we assume that $\V$ is irreducible.   

Since $\hol$ is proper and holomorphic, the intersection $\hol(\Q(Y)) \cap \V_s$ is a properly embedded codimension--zero submanifold of $\V_s$.
Theorem \ref{smooth} tells us that $\V_s$ is connected, and so 
\[
\hol(\Q(Y)) \cap \V_s = \V_s .
\]
Properness implies that $\hol(\Q(Y))$ is closed and since $\V_s$ is dense in $\V$ (see page 124 of \cite{shafarevich2}), we have
\[
\hol(\Q(Y)) = \V .
\]

By proposition \ref{infinite},
\[
\V \cap \T(S) = \hol(\Q(Y)) \cap \T(S) \subset \X_\R(S)
\]
is a countable set of infinitely many points. 
Since $\T(S)$ is a topological component of $\X_\R(S)$, we conclude that $\V \cap \X_\R(S)$ has infinitely many topological components. 
But $\V \cap \X_\R(S)$ is a real algebraic set, which is permitted only a finite number of components, thanks to a theorem of H. Whitney \cite{whitney}.
\end{proof}

Note that the same argument shows that for any open set $\U$ in $\Q(Y)$, the set $\hol(\U)$ is not contained in any subvariety of dimension  $-\frac{3}{2} \chi(Y)$.

\section{Skinning}

\begin{proof}[Proof of Theorem \ref{skin}] Let $\Y_\C^{\, 0}(M)$ denote an irreducible component of $\X_\C(M)$ containing a complete hyperbolic structure on the interior of $M$. 
Let $\X_\C^{\, 0}(M) \subset \Y_\C^{\, 0}(M)$ be the subvariety obtained by demanding all $\Z \oplus \Z$--subgroups to be parabolic.

Let $\partial_0 M$ be the union of the nontorus components of $\partial M$.
Then the complex dimension of $\X_\C^{\, 0}(M)$ is $-\frac{3}{2} \chi(\partial_0 M)$.

First suppose that $\partial_0 M$ is connected.  
The inclusion $\partial_0 M \to M$ induces a regular function
\[
\iota \co \X_\C^{\, 0}(M) \to \X_{\C}(\partial_0 M).
\]
Since regular functions are closed with respect to the Zariski topology, the image $\iota(\X_\C^{\, 0}(M))$ is an algebraic subvariety of $\X_{\C}(\partial_0 M)$.
If $\sigma_M$ were a function identically equal to $Y$, then the Bers slice $\B_Y$ would be contained in $\iota  (\X_\C^{\, 0}(M))$, contradicting Theorem \ref{slice}.

Suppose now that $\partial_0 M$ is disconnected with components $S_0,\ldots,S_n$. 
 For each $i \geq 1$, let $M_i$ be a compact orientable irreducible atoroidal acylindrical $3$--manifold with incompressible boundary homeomorphic to $S_i$, and attach the $M_i$ to the $S_i$ along their boundaries to obtain a manifold $N$.
 This manifold is orientable, irreducible, atoroidal, and any properly embedded essential cylinder is disjoint from the tori in $\partial N$.  
 By Thurston's Geometrization Theorem for Haken manifolds, see \cite{morgansurvey,otal,kapovichbook}, its interior admits a hyperbolic structure without accidental parabolics.
 
The skinning map of $N$ factors 
\[
 \xymatrix{\T(\partial_0 N) \ar[d]_{} \ar[r]^{\sigma_N} &   \T(\overline{\partial_0 N}) \\
  \T(\partial_0 M) \ar[r]_{\sigma_M} & \T(\overline{\partial_0 M}) \ar[u]
  } 
\]
where the vertical map on the left is the map
$
\GF(N) \to \GF(M)
$
induced by inclusion; the one on the right simply projection. We have shown that the skinning map of $N$ is not constant, and it follows that $\sigma_M$ is nonconstant.
\end{proof}

\bibliographystyle{plain}
\bibliography{skinacat}

\bigskip

\noindent Department of Mathematics, Brown University, Providence RI 02912 
\newline \noindent  \texttt{ddumas@math.brown.edu}  , \quad \texttt{rkent@math.brown.edu}

\end{document}